\pdfoutput=1
\RequirePackage{ifpdf}
\ifpdf 
\documentclass[pdftex]{sigma}
\else
\documentclass{sigma}
\fi

\numberwithin{equation}{section}

\newtheorem{Theorem}{Theorem}[section]
\newtheorem{Corollary}[Theorem]{Corollary}
\newtheorem{Lemma}[Theorem]{Lemma}
\newtheorem{Proposition}[Theorem]{Proposition}
 { \theoremstyle{definition}
\newtheorem{Example}[Theorem]{Example}
\newtheorem{Remark}[Theorem]{Remark} }

\newcommand{\N}{{\mathbb Z}_{> 0} }
\newcommand{\Zp}{{\mathbb Z}_{\geq 0} }
\newcommand{\Z}{\mathbb Z}
\newcommand{\C}{\mathbb C}
\newcommand{\HVir}{\mathcal H}
\DeclareMathOperator{\Ker}{Ker}
\DeclareMathOperator{\ch}{char}

\begin{document}

\allowdisplaybreaks

\newcommand{\arXivNumber}{1605.08608}

\renewcommand{\PaperNumber}{113}

\FirstPageHeading

\ShortArticleName{On Free Field Realizations of $W(2,2)$-Modules}

\ArticleName{On Free Field Realizations of $\boldsymbol{W(2,2)}$-Modules}

\Author{Dra\v{z}en ADAMOVI\'C~$^\dag$ and Gordan RADOBOLJA~$^\ddag$}

\AuthorNameForHeading{D.~Adamovi\'c and G.~Radobolja}

\Address{$^\dag$~Department of Mathematics, University of Zagreb, Bijeni\v{c}ka 30, 10 000 Zagreb, Croatia}
\EmailD{\href{mailto:adamovic@math.hr}{adamovic@math.hr}}
\URLaddressD{\url{https://web.math.pmf.unizg.hr/~adamovic/}}

\Address{$^\ddag$~Faculty of Science, University of Split, Rudera Bo\v{s}kovi\'{c}a 33, 21 000 Split, Croatia}
\EmailD{\href{mailto:gordan@pmfst.hr}{gordan@pmfst.hr}}

\ArticleDates{Received June 09, 2016, in f\/inal form December 03, 2016; Published online December 06, 2016}

\Abstract{The aim of the paper is to study modules for the twisted Heisenberg--Virasoro algebra $\mathcal H$ at level zero as modules for the $W(2,2)$-algebra by using construction from~[\textit{J.~Pure Appl.\ Algebra} \textbf{219} (2015), 4322--4342, arXiv:1405.1707]. We prove that the irreducible highest weight ${\mathcal H}$-module is irreducible as $W(2,2)$-module if and only if it has a typical highest weight. Finally, we construct a screening operator acting on the Heisenberg--Virasoro vertex algebra whose kernel is exactly $W(2,2)$ vertex algebra.}


\Classification{17B69; 17B67; 17B68; 81R10}

\section{Introduction}

\looseness=1 Lie algebra $W(2,2)$ was f\/irst introduced by W.~Zhang and C.~Dong in \cite{Zhang-Dong} as part of a classif\/ication of certain simple vertex operator algebras. Its representation theory has been studied in \cite{JP,JZ,LZ1, R-JMP} and several other papers. Although $W(2,2)$ is an extension of the Virasoro algebra, its representation theory is very dif\/ferent. This is most notable with highest weight representations. It was shown in~\cite{R-JMP} that some Verma modules contain a cosingular vector.

Highest weight representation theory of the twisted Heisenberg--Virasoro Lie algebra has also been studied recently. Representations with nontrivial action of $C_{I}$ have been developed in~\cite{Arbarello}. Representations at level zero, i.e., with trivial action of $C_{I}$ were studied in \cite{Billig} due to their importance in some constructions over the toroidal Lie algebras (see \cite{Billig3, Billig2}). In this case, a free f\/ield realization of highest weight modules along with the fusion rules for a suitable category of modules were obtained in~\cite{nas}.

Irreducible highest weight modules of highest weights $(0,0)$ over these algebras carry the structure of simple vertex operator algebras. Let us denote these vertex operator algebras as $L^{W(2,2)} (c_L, c_W)$ and $L^{\HVir} (c_L, c_{L,I} )$. It was proved in~\cite{nas} that simple vertex operator algebra $L^{W(2,2)} (c_L, c_W)$ embeds into Heisenberg--Virasoro vertex operator algebra $L^{\HVir} (c_L, c_{L,I} )$ so that $c_W = -24 c_{L,I} ^2$. As a result each highest weight module over $\HVir$ is also a $W(2,2)$-module. In this paper we shall completely describe the structure of the irreducible highest weight $\HVir$-modules as $W(2,2)$-modules. We show that in generic case the resulting $W(2,2)$-module is irreducible. However, in case of a module of highest weight such that associated Verma module over $W(2,2)$ contains cosingular vectors (we shall call this kind of weight atypical), irreducible $\HVir$-module is reducible over $W(2,2)$.
We shall denote the irreducible highest weight $ \HVir $-module $L^{\HVir}(c_L,0,c_{L,I},h,h_I) $ shortly as $L^{\HVir}(h,h_I)$.
 We also use the following notation\footnote{We emphasise a term $\frac{c_L-2}{24} $ for its importance in a free f\/ield realization of $\HVir $ (see~\cite{nas} for details).}
\begin{gather*}
h_{p,r}=\big(1-p^{2}\big)\frac{c_{L}-2}{24}+p(p-1)+p\frac{1-r}{2} 
\end{gather*}
for $ p,r \in \N $. Def\/ine
\begin{gather*}
{\mathcal {AT} }_{\HVir} ( c_L, c_{L,I} ) = \{ (h_{p,r},(1\pm p)c_{L,I}) \,|\, p,r \in \N \}.
\end{gather*}
We call a weight $(h,h_I)$ \textit{atypical} for $\HVir $ (resp.\ \textit{typical}) if $(h,h_I) \in {\mathcal {AT} } _{\HVir} ( c_L, c_{L,I} ) $ (resp.\ $(h,h_I) \notin {\mathcal {AT} } _{\HVir} ( c_L, c_{L,I} ) $). We shall refer to a highest weight module over $\HVir$ as (a)typical if its highest weight is (a)typical for~$\HVir$.

The next theorem gives a main result of the paper.

\begin{Theorem}\label{main}
Assume that $c_{L,I}\neq0$.
\begin{enumerate}\itemsep=0pt
\item[$(1)$] $L^{\HVir}(h,h_I) $ is irreducible as a $W(2,2)$-module if and only if
\begin{gather*}
(h,h_I) \notin {\mathcal {AT} }_{\HVir} ( c_L, c_{L,I} ) .
\end{gather*}
\item[$(2)$] If $(h,h_I) \in {\mathcal {AT} }_{\HVir} ( c_L, c_{L,I} ) $ then $L^{\HVir}(h,h_I) $ is a non-split extension of two irreducible highest weight $W(2,2)$-modules.
\end{enumerate}
\end{Theorem}

We recall some aspects of representation theories of inf\/inite-dimensional Lie algebras $\HVir$ and~$W(2,2)$ in Section \ref{rep-th-review}. The main results on the branching rules will be proved in Section~\ref{result}. From the free f\/ield realization in~\cite{nas} follows that irreducible $\HVir$-modules are pairwise contragredient. For half of these modules, proofs rely on a~$W(2,2)$-homomorphism between Verma modules over~$W(2,2)$ and $\HVir$ which is induced by a homomorphism of vertex operator algebras. The rest is then proved elegantly by passing to contragredients. We also prove a very interesting result that the Verma module for $\HVir$ with typical highest weight is an inf\/inite direct sum of irreducible $W(2,2)$-modules (cf.\ Theorem~\ref{t37}). This result presents a~$W(2,2)$-analogue of certain Feigin--Fuchs modules for the Virasoro algebra (cf.\ Remark~\ref{ff}).

From the results in the paper, we see that the vertex algebra $L ^{W(2,2)} (c_L, c_W )$ has many properties similar to the $\mathcal W$-algebras appearing in logarithmic conformal f\/ield theory (LCFT):
\begin{itemize}\itemsep=0pt
\item $L ^{W(2,2)} (c_L, c_W )$ admits a free f\/ield realization inside of the Heisenberg--Virasoro vertex algebra $L^{\HVir} (c_L, c_{L,I})$.
\item Typical modules are realized as irreducible modules for $L^{\HVir} (c_L, c_{L,I})$.
\item In the atypical case, irreducible $L^{\HVir} (c_L, c_{L,I})$-modules as $L ^{W(2,2)} (c_L, c_W )$-modules have semi--simple rank two.
 \end{itemize}
The singlet vertex algebra $\overline{M(1)}$ has similar properties. $\overline{M(1)}$ is realized as kernel of a screening operator inside the Heisenberg vertex algebra $M(1)$ (cf.~\cite{A1}). In Section~\ref{screenings} we construct the screening operator
\begin{gather*}
S _1\colon \ L^ {\HVir} (c_L, c_{L,I} ) \rightarrow L^{\HVir} (1,0),
\end{gather*}
which commutes with the action of $W(2,2)$-algebra such that
\begin{gather*}
\Ker_ { L^{\HVir} (c_L,c_{L,I}) } S_1 \cong L ^{W(2,2)} (c_L, c_W ) .
\end{gather*}
Our construction uses an extension $\mathcal V_{\rm ext}$ of the vertex algebra $L^ {\HVir} (c_L, c_{L,I} )$ by a non-weight module for the Heisenberg--Virasoro vertex algebra. In our forthcoming paper~\cite{AR2}, we shall present an explicit realization of $\mathcal V_{\rm ext}$ and apply this construction to the study of intertwining operators and logarithmic modules.

\section[Lie algebra $W(2,2)$ and the twisted Heisenberg--Virasoro Lie algebra at level zero]{Lie algebra $\boldsymbol{W(2,2)}$ and the twisted Heisenberg--Virasoro\\ Lie algebra at level zero}\label{rep-th-review}
$W(2,2)$ is a Lie algebra with basis $\{L(n), W(n),C_{L},C_{W}\colon n\in\mathbb{Z\}}$ over $\C$, and a Lie bracket
\begin{gather*}
[ L(n) ,L(m) ] =(n-m) L(n+m) +\delta_{n,-m}\frac{n^{3}-n}{12}C_{L}, \\ 
[ L(n) ,W(m) ] =(n-m)W(n+m) +\delta_{n,-m}\frac{n^{3}-n}{12}C_{W}, \\
[ W(n) ,W(m) ] =[ \cdot,C_{L}] =[ \cdot,C_{W}] =0.
\end{gather*}

Highest weight representation theory over $W(2,2)$ was studied in~\cite{JP,R-JMP}. However, representations treated in these papers have
equal central charges $C_{L}=C_{W}$. These results have recently been generalised to $C_{L}\neq C_{W}$ in~\cite{JZ}. Here we state the most important results. Verma module with central charge $(c_{L},c_{W}) $ and highest weight $(h,h_{W}) $ is denoted by $V^{W(2,2)}( c_{L},c_{W},h,h_{W})$, its highest weight vector
by $v_{h,h_W}$ and irreducible quotient module by $L^{W(2,2)} (c_{L},c_{W},h,h_{W})$.

Recall the def\/inition of a cosingular vector. Homogeneous vector $v\in M$ is called cosingular (or subsingular) if it is not singular in $M$ and if there
is a proper submodule $N \subset M$ such that $v+N$ is a singular vector in~$M/N$.

\begin{Theorem}[{\cite{JZ, R-JMP}}]\label{W-struktura}Let $c_{W}\neq0$.
\begin{enumerate}\itemsep=0pt
\item[$(i)$]Verma module $V^{W(2,2)}(c_{L},c_{W},h,h_{W}) $ is reducible if and only if $h_{W}=\frac{1-p^{2}}{24}c_{W}$ for some $p\in \N$. In that case, there exists a singular vector $u_p^{\prime}\in \C [ W(-1) ,\ldots,W(-p)]v_{h,h_W} $ such that $U(W(2,2)) u_p^{\prime} \cong V^{W(2,2)}(c_{L},c_{W},h+p,h_{W}) $.

\item[$(ii)$]A quotient module\footnote{This module is denoted by $L^{\prime}$ in \cite{JZ,R-JMP}. We
change notation to $\widetilde{L}$ due to use of superscript $W(2,2)$.}
\begin{gather*}V^{W(2,2)}(c_{L},c_{W},h,h_{W}) / U(W(2,2)) u_p^{\prime} =: \widetilde{L}^{W(2,2)}( c_{L},c_{W},h_{p,r},h_{W})
\end{gather*}
is reducible if and only if $ h = h_{p,r} $ for some $r\in\N$. In that case, there is a cosingular vector $%
u_{rp}\in V^{W(2,2)}(c_{L},c_{W},h,h_{W}) _{h+rp}$ such that $\overline{u_{rp}} := u_{rp} + U(W(2,2)) u_p^{\prime}$ is a~singular vector in $\widetilde{L}^{W(2,2)}( c_{L},c_{W},h_{p,r},h_{W})$ which generates a submodule isomorphic to $L^{W(2,2)}( c_{L},c_{W},h_{p,r}+rp,h_{W})$. The short sequence
\begin{align}
0 &\rightarrow L^{W(2,2)}(c_{L},c_{W},h_{p,r}+rp,h_{W}) \rightarrow\widetilde{L}^{W(2,2)} ( c_{L},c_{W},h_{p,r},h_{W} ) \nonumber \\
&\rightarrow L^{W(2,2)} ( c_{L},c_{W},h_{p,r},h_{W} ) \rightarrow 0, \label{seq}
\end{align}
where the highest weight vector in $L^{W(2,2)}(c_{L},c_{W},h_{p,r}+rp,h_{W})$ maps to $\overline{u_{rp}}$ is exact.
\end{enumerate}
\end{Theorem}

Def\/ine
\begin{gather*}
{\mathcal {AT} }_{W(2,2) } ( c_L, c_{W} ) = \left\{ \left(h_{p,r}, \frac{1-p^2}{24}c_W\right) |\, p,r \in \N \right\}.
\end{gather*}
\begin{Remark}
We will refer to the (modules of) highest weights $(h, h_W) \in {\mathcal {AT} }_{W(2,2) } ( c_L, c_{W} ) $ as \textit{atypical} for $W(2,2)$, and otherwise as \textit{typical}. Again, we refer to a highest weight $W(2,2)$-module as (a)typical depending on its highest weight. So a Verma module over $W(2,2)$ contains a nontrivial cosingular vector if and only if it is atypical.
\end{Remark}

\begin{Proposition} \label{svi cosing}
Let $h_W=\frac{1-p^2}{24}c_W$, $p\in\N$.
\begin{enumerate}\itemsep=0pt
\item[$(i)$]Let $(h_{p,r},h_W)$, $r\in\N$ be an atypical weight and $k\in\Z$. Then $(h_{p,r}+kp,h_W)$ is atypical if and only if $k<\frac{r}{2}$.
\item[$(ii)$]Atypical Verma module $V^{W(2,2)}(h_{p,r},h_W)$ contains exactly $\lfloor \frac{r+1}{2} \rfloor$ cosingular vectors. The weights of these vectors are $h_{p,r}+(r-i)p = h_{p,-r+2i}$, $i=0,\ldots,\lfloor \frac{r-1}{2} \rfloor$.
\end{enumerate}
\end{Proposition}

\begin{proof}
(i) Directly from Theorem~\ref{W-struktura} since $h_{p,r}+kp = h_{p,r-2k}$.

(ii) Follows from~(i) since $V^{W(2,2)}(h_{p,r},h_W)$ contains an inf\/inite chain of submodules isomorphic to Verma modules of highest weights $h_{p,r}+ip = h_{p,r-2i}$, $i>0$. Applying Theorem~\ref{W-struktura} to each of these submodules we obtain cosingular vectors of weights
\begin{gather*}
h_{p,r-2i}+(r-2i)p = h_{p,r}+(r-i)p = h_{p,-r+2i}
\end{gather*}
as long as $r-2i>0$.
\end{proof}

\begin{Remark}\label{pbw}
Standard PBW basis for $V^{W(2,2)}(c_{L},c_{W},h,h_{W}) $ consists of vectors
\begin{gather*}
W(-m_{s}) \cdots W(-m_{1}) L(-n_{t}) \cdots L ( -n_{1} ) v_{h,h_W}
\end{gather*}
such that $ m_{s}\geq\cdots\geq m_{1}\geq1$, $n_{t}\geq\cdots\geq n_{1}\geq1 $. The only nonzero component of $u_{rp}$ belonging to $\C [ L(-1) ,L ( -2), \ldots ] v$ is $L(-p)^{r}v_{h,h_W}$~\cite{R-JMP}.
\end{Remark}

Def\/ine $P_{2}(n) =\sum\limits_{i=0}^{n}P(n-i)P(i)$ where $P $ is a partition function with $P(0)=1$. We have the following character
formulas~\cite{R-JMP}
\begin{gather*}
\ch V^{W(2,2)}( c_{L},c_{W},h,h_{W}) =q^{h}\sum_{n\geq0}P_{2}(n)q^{n}=q^{h}\prod\limits_{k\geq1}\big(1-q^{k}\big)^{-2},
\end{gather*}
for all $h,h_W \in \C$. If $h_W = \frac{1-p^2}{24}c_W$, then
\begin{gather*}
\ch \widetilde{L}^{W(2,2)}(c_{L},c_{W},h,h_{W}) =q^{h}\big(1-q^{p}\big)\sum_{n\geq0}P_{2}(n)q^{n}
 =q^{h}\big(1-q^{p}\big)\prod\limits_{k\geq1}\big(1-q^{k}\big)^{-2}.
\end{gather*}
If $(h,h_W)$ is typical for $W(2,2)$, then this is the character of an irreducible highest weight module. Finally, the character of atypical irreducible module is
\begin{align*}
\ch L^{W(2,2)}( c_{L},c_{W},h_{p,r},h_{W})& =q^{h_{p,r}}\big(1-q^{p}\big)\big(1-q^{rp}\big)\sum_{n\geq0}P_{2}(n)q^{n} \\
& =q^{h_{p,r}}\big(1-q^{p}\big)(1-q^{rp})\prod\limits_{k\geq1}\big(1-q^{k}\big)^{-2}.
\end{align*}

The twisted Heisenberg--Virasoro algebra $\HVir$ is the universal central extension of the Lie algebra of dif\/ferential operators on a circle of order at most one. It is the inf\/inite-dimensional complex Lie algebra with a basis
\begin{gather*}
\{L(n),I(n)\colon n\in\Z\}\cup\{C_{L},C_{LI},C_{I}\}
\end{gather*}
and commutation relations
\begin{gather*}
[ L(n),L(m)] =(n-m)L(n+m)+\delta_{n,-m}\frac{n^{3}-n}{12}C_{L},\\
[ L(n),I(m)] =-mI(n+m)-\delta_{n,-m}\big(n^{2}+n\big)C_{LI}, \\
[ I(n),I(m)] =n\delta_{n,-m}C_{I}, \qquad
[ {\HVir},C_{L}] =[ {\HVir},C_{LI}] =[ \HVir,C_{I}] =0.
\end{gather*}
The Lie algebra $\HVir$ admits the following triangular decomposition
\begin{gather}
\HVir =\HVir ^{-} \oplus \HVir^0 \oplus \HVir^{+}, \label{triangular} \\
\HVir ^{\pm} = \operatorname{span}_{\mathbb C} \{ I( \pm n), L(\pm n) \,|\, n \in {\Z}_{>0}\}, \qquad
\HVir^0 = \operatorname{span}_{\mathbb C} \{ I(0), L(0), C_L, C_{L,I}, C_{I} \}. \nonumber
\end{gather}

Although they seem to be two similar extensions of the Virasoro algebra, representation theories of $W(2,2)$ and $\HVir$ are dif\/ferent. The main reason for that lies in the fact that $I(0) $ is a~central element, while $W(0) $ is not. However, applying free f\/ield realization, we shall see that highest weight modules over the two algebras are related.

Denote by $V^{\HVir} ( c_{L},c_{I},c_{L,I},h,h_{I} ) $ the Verma module and by $v_{h,h_I}$ its highest weight vector. $C_{L},$ $C_{I},$ $C_{L,I},$ $L(0) $ and $I(0) $ act on $v_{h,h_I}$ by scalars $c_{L}$, $c_{I}$, $c_{L,I}$, $h$ and $h_{I}$, respectively. Then $\left(c_{L,}c_{I},c_{L,I}\right) $ is called a central charge, and $(h,h_{I}) $ a highest weight. In this paper we consider central charges $(c_{L},0,c_{L,I}) $ such that $c_{L,I}\neq0$.

\begin{Theorem}[\cite{Billig}]\label{billig}
Let $c_{L,I}\neq0$. Verma module $V^{\HVir}(c_{L},0,c_{L,I},h,h_{I}) $ is reducible if and only if $h_{I}=(1\pm p) c_{L,I}$ for some $p\in\N$. In that case, there is a singular vector $v_p^{\pm}$ of weight $p$, which generates a maximal submodule in $V^{\HVir} ( c_{L},0,c_{L,I},h,h_{I} ) $ isomorphic to $V^{\HVir} ( c_{L},0,c_{L,I},h+p,h_{I} ) $.
\end{Theorem}

\begin{Remark} \label{kompon}
In case $h_{I}=(1+p) c_{L,I}$ an explicit formula for a singular vector $v_p^+$ is obtained using Schur polynomials in $I(-1), \ldots, I(-p)$. See~\cite{nas} for details. Assume that $x \in U(W(2,2))_-$ is such that $xv_{h,h_I} \in V^{\HVir} ( c_{L},0,c_{L,I},h,h_{I} ) $ lies in a~maximal submodule. Then $x$ does not have a nontrivial additive component (in PBW basis) that belongs to $\C [ L(-1) ,L(-2), \ldots ]$~\cite{Billig}.
\end{Remark}

There is an inf\/inite chain of Verma submodules generated by singular vectors $v_{kp}^{\pm}$, $k\in\N$, with all the subquotients being irreducible. Note that there is no mention of $\widetilde{L}^{\HVir}$ since there are no cosingular vectors in $V^{\HVir}$.

The following character formulas were obtained in \cite{Billig}:
\begin{gather*}
\ch V^{\HVir}( c_{L},0,c_{L},h,h_{I}) =q^{h}\sum_{n\geq0}P_{2}(n)q^{n}=q^{h}\prod\limits_{k\geq1}\big(1-q^{k}\big)^{-2}, \\
\ch L^{\HVir} ( c_{L},0,c_{L},h,h_{I} ) =q^{h}\big(1-q^{p}\big)\sum_{n\geq0}P_{2}(n)q^{n} =q^{h}\big(1-q^{p}\big)\prod\limits_{k\geq1}\big(1-q^{k}\big)^{-2}.
\end{gather*}

\begin{Remark}
Throughout the rest of the paper we work with highest weight modules over the Lie algebras $W(2,2)$ and $\HVir$ so we always denote algebra in superscript. In order to avoid too cumbersome notation, we omit central charges. Therefore, we write $V^{\HVir}(h,h_{I}) $ for Verma module over~$\HVir$, ${V}^{W(2,2)}(h,h_{W}) $ for Verma module over $W(2,2)$ and so on. We always assume that $c_{W}$ and $c_{L,I}$ are nonzero. Moreover, if we work with several modules over both algebras, $c_{L}$ is equal for all modules.

We shall write $\langle x \rangle_{W(2,2)}$ for a cyclic submodule $U (W(2,2))x$ and $\langle x \rangle_{\HVir}$ for $U(\HVir)x$. Finally, $\cong_{W(2,2)}$ denotes an isomorphism of $W(2,2)$-modules.
\end{Remark}

\section{Irreducible highest weight modules}\label{result}
In this section we present main results of the paper which completely describe the structure of (irreducible) highest weight modules for $\HVir$ as $W(2,2)$-modules. The main tool is the homomorphism between $W(2,2)$ and the Heisenberg--Virasoro vertex algebras from~\cite{nas}.

$L^{W(2,2) }( c_{L},c_{W},0,0) $ is a simple universal vertex algebra associated to Lie algebra $W(2,2)$ (cf.\ \cite{R-JMP, Zhang-Dong}) which we denote by $L^{W(2,2)}(c_{L},c_{W})$. It is generated by f\/ields
\begin{gather*}
L(z)=Y(\omega ,z)=\sum_{n\in {\Z}}L(n)z^{-n-2},\qquad W(z)=Y(W,z)=\sum_{n\in {\Z}}W(n)z^{-n-2},
\end{gather*}
where $\omega =L(-2)\mathbf{1}$ and $W=W(-2)\mathbf{1}$. Each highest weight $W(2,2) $-module is also a module over a vertex operator algebra $L^{W(2,2) }(c_{L},c_{W})$.

Likewise (see \cite{Billig3}) $L^{\HVir}(c_{L},0,c_{L,I},0,0) $ is a simple Heisenberg--Virasoro vertex operator algebra, which we denote by $L^{\HVir}(c_{L},c_{L,I}) $. This algebra is generated by the f\/ields
\begin{gather*}
L(z)=Y(\omega ,z)=\sum_{n\in {\Z}}L(n)z^{-n-2},\qquad I(z)=Y(I,z)=\sum_{n\in {\Z}}I(n)z^{-n-1},
\end{gather*}
where $\omega =L(-2)\mathbf{1}$ and $I=I(-1)\mathbf{1}$. Moreover, highest weight $\HVir$-modules are modules over a~vertex operator algebra
$L^{\HVir}(c_{L,}c_{L,I})$.

It was shown in \cite{nas} that there is a monomorphism of vertex operator algebras
\begin{align}
\Psi \colon \ L^{W(2,2)}(c_{L},c_{W})& \rightarrow L^{\HVir}(c_{L},c_{L,I}),\label{Psi} \\
\omega & \mapsto L(-2)\mathbf{1}, \notag \\
W& \mapsto (I(-1)^{2}+2c_{L,I}I(-2))\mathbf{1}, \notag
\end{align}%
where $c_{W}=-24c_{L,I}^{2}$. By means of $\Psi $, each highest weight module over $\HVir$ becomes an \linebreak $L^{W(2,2)}(c_{L},c_{W})$-module and
therefore a module over $W(2,2) $. In particular, $\Psi $ induces a non-trivial $W(2,2)$-homomorphism (which we shall denote by the same letter)
\begin{gather*}
\Psi \colon \ V^{W(2,2)}(c_{L},c_{W},h,h_{W}) \rightarrow V^{\HVir} ( c_{L},0,c_{L,I},h,h_{I} ),
\end{gather*}
where $c_{W}=-24c_{L,I}^{2}$ and $h_{W}=h_{I}(h_{I}-2c_{L,I})$. $\Psi$ maps the highest weight vector $v_{h,h_{W}}$ to the highest weight vector
$v_{h,h_{I}}$ and the action of $W(-n)$ on $V^{\HVir} ( c_{L},0,c_{L,I},h,h_{I} )$ is given by
\begin{gather}
W(-n) \equiv 2c_{L,I}(n-1) I(-n)+\sum_{i\in \Z}I(-i)I(-n+i),\label{W u I} \\
W(-n) \equiv 2 c_{L,I}\left(n- 1 + \frac{h_I}{c_{L,I} } \right) I(- n)+\sum_{i\ne 0,n } I(-i)I(- n+i).\nonumber
\end{gather}

Note that $h_{W}=\frac{1-p^{2}}{24}c_{W}$ if and only if $h_{I}= ( 1\pm p) c_{L,I}$, so either both of these Verma modules are irreducible, or they are reducible with singular vectors at equal levels. Moreover, $ (h,h_W) \in {\mathcal {AT} }_{W(2,2)} ( c_L, c_W ) $ if and only if $(h,h_I) \in {\mathcal {AT} }_{\HVir} ( c_L, c_{L,I} )$.

Throughout the rest of this section we assume that $ c_W = -24c_{L,I}^2 $.

\begin{Lemma}[{\cite[Lemma 7.2]{nas}}] \label{psi}Suppose that $h_{I}\neq(1-p) c_{L,I}$ for all $p\in \N$. Then $\Psi$ is an isomorphism of $W(2,2)$-modules. In particular, if $ h_I \neq (1\pm p) c_{L,I} $ for $p \in \N $, then
\begin{gather*}
L^{\HVir}(h,h_I)\cong_{W(2,2)}L^{W(2,2)}(h,h_W),
\end{gather*}
where $h_W = h_I (h_I-2c_{L,I}) $.
\end{Lemma}

\begin{Lemma}\label{H-sing je W-sing}Suppose that $x\in V^{\HVir}(h,h_I)$ is $\HVir$-singular. Then $x$ is $W(2,2)$-singular as well. In particular, if $y$ is a homogeneous vector such that $x = \Psi(y)$, then $y$ is either singular or cosingular vector in $V^{W(2,2)}(h,h_W)$.
\end{Lemma}

\begin{proof}
Let $x\in V^{\HVir}(h,h_I)$ be a $\HVir$-singular vector, i.e., $L(k)x=I(k)x=0$ for all $k\in\N$. From (\ref{W u I}) we have
\begin{gather*}
W(n)x = -2c_{L,I}(n+1) I(n)x+\sum_{i\in \Z}I(-i)I(n+i)x,
\end{gather*}
so $W(n)x=0$ for all $n\in\N$. Therefore, $x$ is $W(2,2)$-singular.
If $x=\Psi(y)$, then $L(k)y, W(k)y$ $\in \Ker \Psi$ for $k>0$. Therefore $y + \Ker \Psi$ is a singular vector in $V^{W(2,2)}(h,h_W) / \Ker \Psi$.
\end{proof}

\begin{Theorem} \label{t33} \label{L za +}Let $p\in\N$.
\begin{enumerate}\itemsep=0pt
\item[$(i)$] If $(h,(1+p) c_{L,I})$ is typical for $\HVir$ $($equivalently if $\big(h,\frac{1-p^{2}}{24}c_{W}\big) $ is typical for $W(2,2))$ then
\begin{gather}
L^{\HVir}(h,(1+p) c_{L,I}) \cong_{W(2,2) }L^{W(2,2)}\left( h,\frac{1-p^{2}}{24}c_{W}\right).\label{L + gen}
\end{gather}

\item[$(ii)$] If $(h_{p,r},(1+p) c_{L,I}) \in {\mathcal {AT} }_{\HVir} ( c_L, c_{L,I} )$ $($equivalently if $ (h_{p,r},\frac{1-p^{2}}{24}c_{W}) \in {\mathcal {AT} }_{W(2,2)} ( c_L, c_W ))$ then
\begin{gather*}
L^{\HVir}(h_{p,r},(1+p) c_{L,I}) \cong_{W(2,2) }\widetilde{L}^{W(2,2)}\left( h_{p,r},\frac{1-p^{2}}{24}c_{W}\right) 
\end{gather*}
and the short sequence of $W(2,2)$-modules
\begin{align}
0 & \rightarrow L^{W(2,2)}\left( h_{p,r}+rp,\frac{1-p^{2}}{24}c_{W}\right)\rightarrow L^{\HVir}(h_{p,r},(1+p) c_{L,I}) \label{seq +} \\
& \rightarrow L^{W(2,2)}\left( h_{p,r},\frac{1-p^{2}}{24}c_{W}\right) \rightarrow0 \nonumber
\end{align}
is exact.
\end{enumerate}
\end{Theorem}

\begin{proof}
By Lemma \ref{psi}, $\Psi$ is an isomorphism of Verma modules and thus by Lemma \ref{H-sing je W-sing} it maps a $W(2,2) $-singular vector $u^{\prime}_p$ to an $\HVir$-singular vector $v_p^+$. If $h\neq h_{p,r}$, both of these vectors generate maximal submodules in respective Verma modules so~(\ref{L + gen}) follows.

Now suppose that $h=h_{p,r}$. We need to show that a cosingular vector $u_{rp}$ is not mapped into a~maximal submodule of $V^\HVir(h_{p,r},h_I)$. But $u_{rp}$ has $L(-p) ^{r}v$ as an additive component (see Remark~\ref{pbw}), and by construction~(\ref{Psi}), $\Psi(u_{rp})$ also must have this additive component. However, $\Psi(u_{rp})$ can not lie in a maximal $\HVir$-submodule of $V^{\HVir}(h,h_{I}) $ (see Remark~\ref{kompon}). This means that isomorphism $\Psi$ of Verma modules induces a $W(2,2) $-isomorphism of $\widetilde{L}^{W(2,2)}(h,h_{W}) $ and $L^{\HVir}(h,h_{I}) $ for all $h\in\C$. Exactness of~(\ref{seq +}) is just an application of~(\ref{seq}).
\end{proof}

\begin{Remark}
Note that the image $\Psi(u_{rp}) $ of a $W(2,2)$-cosingular vector is neither $\HVir$-singular, nor $\HVir$-cosingular in $V^{\HVir}(h_{p,r},(1+p)c_{L,I})$. For example, $L(-1) v_{0,0}$ in $V^{\HVir}(0,2c_{L,I})$ is $W(2,2)$-cosingular, but not $\HVir$-singular since $I(1) L(-1) v_{0,0}=2c_{L,I}v_{0,0}$.
\end{Remark}

If $h_{I}=(1-p) c_{L,I}$, then $\Psi $ is not an isomorphism. We shall present a $W(2,2)$-structure of Verma module later. In order to examine irreducible $W(2,2)$-modules we apply the properties of contragredient modules.

Let us recall the def\/inition of contragredient module (see~\cite{FHL}). Assume that $(M,Y_{M})$ is a~graded module over a vertex operator algebra~$V$ such that $M=\oplus _{n=0}^{\infty }M(n)$, $\dim M(n)<\infty $ and suppose that there is $\gamma \in {\C}$ such that $L(0)|M(n)\equiv (\gamma +n)\operatorname{Id}$. The contragredient module $( M^{\ast },Y_{M^{\ast}}) $ is def\/ined as follows. For every $n\in {\N}$ let $M(n)^{\ast }$ be the dual vector space and let $M^{\ast }=\oplus _{n=0}^{\infty }M(n)^{\ast }$ be a restricted dual of~$M$. Consider the natural pairing $\langle \cdot ,\cdot \rangle :M^{\ast }\otimes M\rightarrow \C$. Def\/ine the linear map $Y_{M^{\ast }}\colon V\rightarrow \operatorname{End} M^{\ast }[[z,z^{-1}]]$ such that
\begin{gather}
\langle Y_{M^{\ast }}(v,z)m^{\prime },m\rangle =\big\langle m^{\prime },Y_{M}\big(e^{zL(1)}\big({-}z^{-2}\big)^{L(0)}v,z^{-1}\big)m\big\rangle \label{contra}
\end{gather}
for each $v\in V$, $m\in M$, $m^{\prime }\in M^{\ast }$. Then $(M^{\ast },Y_{M^{\ast }}) $ is a $V$-module.

In particular, choosing $v=\omega =L_{-2}\boldsymbol{1}$ in (\ref{contra}) one gets
\begin{gather*}
\langle L(n) m^{\prime },m\rangle =\langle m^{\prime },L(-n) m\rangle .
\end{gather*}
Simple calculation with $I\in L^{\HVir}(c_{L},c_{L,I}) $ and $W\in L^{W(2,2) } ( c_{L},c_{W} ) $ shows that
\begin{gather*}
\langle I(n) m^{\prime },m\rangle = \langle m^{\prime },-I(-n) m+\delta _{n,0}2c_{L,I} \rangle , \qquad
\langle W(n) m^{\prime },m\rangle =\langle m^{\prime },W(-n) m\rangle .
\end{gather*}
Therefore we get the following result (the f\/irst and third relations were given in~\cite{nas}):

\begin{Lemma}\label{duali}
\begin{gather*}
L^{\HVir}(h,h_{I})^{\ast} \cong L^{\HVir}(h,-h_{I}+2c_{L,I}),\qquad
L^{W(2,2) }(h,h_{W})^{\ast} \cong L^{W(2,2) }(h,h_{W}) .
\end{gather*}
In particular,
\begin{gather*}
L^{\HVir}(h,(1\pm p) c_{L,I})^{\ast}\cong L^{\HVir}(h,(1\mp p) c_{L,I}).
\end{gather*}
\end{Lemma}

Directly from Theorem \ref{L za +} and Lemma \ref{duali} follows

\begin{Corollary}\label{L za -}Let $p \in \N $.
\begin{enumerate}\itemsep=0pt
\item[$(i)$] If $(h,(1-p) c_{L,I})$ is typical for $\HVir$ $($equivalently if $\big( h,\frac{1-p^{2}}{24}c_{W}\big) $ is typical for $W(2,2))$ then
\begin{gather*}
L^{\HVir} ( h,(1-p) c_{L,I} ) \cong_{W(2,2) }L^{W(2,2)}\left( h,\frac{1-p^{2}}{24}c_{W}\right) .
\end{gather*}

\item[$(ii)$] If $(h_{p,r},(1-p) c_{L,I}) \in {\mathcal {AT} }_{\HVir} ( c_L, c_{L,I} )$ $($equivalently if $ \big(h_{p,r},\frac{1-p^{2}}{24}c_{W}\big) \in {\mathcal {AT} }_{W(2,2)} ( c_L, c_W ) )$ then
\begin{gather*}
L^{\HVir} ( h_{p,r},(1-p) c_{L,I} ) \cong_{W(2,2) }\widetilde{L}^{W(2,2)}\left( h_{p,r},\frac{1-p^{2}}{24}c_{W}\right) ^{\ast}
\end{gather*}
and the short sequence of $W(2,2)$-modules
\begin{align*}
0 & \rightarrow L^{W(2,2)}\left( h_{p,r},\frac{1-p^{2}}{24}c_{W}\right)
\rightarrow L^{\HVir}\left( h_{p,r},(1-p) c_{L,I}\right) \\ 
& \rightarrow L^{W(2,2)}\left( h_{p,r}+rp,\frac{1-p^{2}}{24}c_{W}\right)
\rightarrow0 \nonumber
\end{align*}
is exact.
\end{enumerate}
\end{Corollary}

From Lemma \ref{psi}, Theorem \ref{L za +} and Corollary \ref{L za -} follow assertions of Theorem \ref{main}.

Finally, we show that Verma module over $\HVir$ is an inf\/inite direct sum of irreducible $W(2,2)$-modules. Recall that $V^{\HVir} ( h,(1-p) c_{L,I} )$ has a series of singular vectors $v_{ip}^-$, $i\in \Z_{\ge 0} $ (for $i= 0$, we set $v^-_0 = v_{h, h_I}$) which generate a descending chain of Verma submodules over~$\HVir$:
\begin{align*}
\langle v_{h,h_I} \rangle_{\HVir} =& V^{\HVir}(h,h_I) \\
\rotatebox{90}{$\subseteq$}\qquad & \\[-2mm]
\langle v_p^- \rangle_{\HVir} \cong& V^{\HVir}(h+p,h_I) \\
\rotatebox{90}{$\subseteq$}\qquad & \\[-3mm]
\vdots \qquad & \\
\rotatebox{90}{$\subseteq$}\qquad & \\[-3mm]
\langle v_{ip}^- \rangle_{\HVir} \cong& V^{\HVir}(h+ip,h_I) \\
\rotatebox{90}{$\subseteq$}\qquad & \\[-3mm]
\langle v_{(i+1)p}^- \rangle_{\HVir} \cong& V^{\HVir}(h+(i+1)p,h_I) \\
\rotatebox{90}{$\subseteq$}\qquad & \\[-3mm]
\vdots \qquad &
\end{align*}
Therefore one may identify $V^{\HVir}(h+ip,h_I)$ with a submodule of $V^{\HVir}(h,h_I) $ and a singular vector $v_{ip}^-\in V^{\HVir}(h,h_I)$ with the highest weight vector $v_{h+ip,h_I} \in V^{\HVir}(h+ip,h_I)$. We will prove that in a typical case each of those vectors generates an irreducible $W(2,2)$-submodule.

\begin{Theorem} \label{t37}
Let $p\in\N$. Suppose that $(h,(1-p) c_{L,I}) \notin {\mathcal {AT} }_{\HVir} (c_L, c_{L,I} )$.
Then we have the following isomorphism of $W(2,2)$-modules
\begin{gather*}
V^{\HVir}(h,(1-p) c_{L,I}) \cong_{W(2,2)}\bigoplus\limits_{i\geq0}L^{W( 2,2)}\left(h+ip,\frac{1-p^{2}}{24}c_{W}\right).
\end{gather*}
\end{Theorem}

\begin{proof} First we notice that the vertex algebra homorphism $\Psi \colon L^{W(2,2)} (c_W, c_L) \rightarrow L^{ \HVir } (c_W, c_L) $, for every $i \in {\Z}_{\ge 0}$ induces the following non-trivial homomorphism of $W(2,2)$-modules:
\begin{gather*}
\Psi ^{(i)} \colon \ V^{W( 2,2)}\left(h+ip,\frac{1-p^{2}}{24}c_{W}\right) \rightarrow \langle v_{ip}^- \rangle_{W(2,2) } \subset V^{\HVir}(h+ip,(1-p) c_{L,I}) ,
\end{gather*}
which maps the highest weight vector of $V^{W( 2,2)}\big(h+ip,\frac{1-p^{2}}{24}c_{W}\big) $ to $v_{ip}^- $. Since $(h,\frac{1-p^{2}}{24}c_{W})$ is typical it follows from Proposition~\ref{svi cosing}(i) that $(h+ip,\frac{1-p^{2}}{24}c_{W})$ are typical for all $i\in\N$ as well.

Let $h_W=\frac{1-p^2}{24}c_W$. Consider the homomorphism $\Psi ^{(i)} \colon V^{ W(2,2) }(h+ ip ,h_W) \rightarrow V^{\HVir}(h+ip,h_I) $ above. Applying (\ref{W u I}), we get
\begin{gather*}
\Psi^{(i)} (W(-p) v_{h+ip ,h_{W}}) =\sum_{i=1}^{p-1}I(-i) I(i-p) v_{h+ip ,h_{I}},
\end{gather*}
so $I(-p)v_{h+ip ,h_I} \notin\operatorname{Im} \Psi ^{(i )} $. Since the Verma modules $V^{W(2,2)}(h+ip ,h_W) $ and $V^{\HVir}(h+i p ,h_I)$ have equal characters, it follows that $\Ker \Psi^{(i )} $ contains a singular vector in $V^{ W(2,2) }(h+ i p ,h_W) $ of conformal weight $ h + (i+1) p$.
Since the weight $(h+ip, h_W)$ is typical, the maximal submodule in $V^{W(2,2)}(h+i p ,h_{W}) $ is generated by this singular vector so we conclude that $\Ker \Psi ^{(i )} $ is the maximal submodule in $V^{W(2,2)}(h+i p ,h_{W}) $. Therefore
\begin{gather*}
\operatorname{Im} \Psi ^{(i )} = \langle v_{h+ ip ,h_{I}} \rangle _{W(2,2)}\cong L^{W(2,2)}(h+ i p ,h_W).
\end{gather*}
In this way we get a series of $W(2,2)$-monomorphisms
\begin{gather}
L^{W(2,2)}(h+ip,h_{W}) \hookrightarrow V^{\HVir}(h,(1-p) c_{L,I}) ,\qquad i\in\Zp \label{ulaganja}
\end{gather}
mapping $v_{h+ip,h_{W}}$ to a singular vector $v_{ip}^-$. Let $v_{j p}^-$ be an $\HVir$-singular vector in $V^{\HVir}(h + i p ,(1-p) c_{L,I})$ of weight $h + j p$, for $j > i$. By Lemma~\ref{H-sing je W-sing}, $v_{j p} ^- $ is singular for $W(2,2)$ and therefore $v_{j p}^- \notin \left\langle v_{h+ i p ,h_{I}}\right\rangle _{W(2,2)}$ for $j > i$. We conclude that the images of morphisms~(\ref{ulaganja}) have trivial pairwise intersections (since these images are non-isomorphic irreducible modules), so their sum is direct. The assertion follows from the observation that the character of this sum is
\begin{gather*}
\sum_{i =0} ^{\infty} q^{h+ i p }\big(1-q^{p}\big)\prod\limits_{k\geq1}\big(1-q^{k}\big)^{-2} = q^{h } \prod\limits_{k\geq1}\big(1-q^{k}\big)^{-2} = \operatorname{char} V^{\HVir} (h, (1-p)c_{L.I} ). \tag*{\qed}
\end{gather*}
\renewcommand{\qed}{}
\end{proof}

\begin{Remark} \label{ff}
It is interesting to notice that our Theorem~\ref{t37} shows that $V^{\HVir}(h,h_I) $ can be considered as a $W(2,2)$-analogue of certain Feigin--Fuchs modules for the Virasoro algebra which are also direct sums of inf\/initely many irreducible modules (cf.\ \cite{FF}, \cite[Theorem~5.1]{AdM1}).
\end{Remark}

From the previous theorem follows
\begin{align*}
V^{\HVir}(h,h_I) \\
\rotatebox{90}{$=$} \qquad & \\[-1mm]
\langle v_{h,h_I} \rangle_{W(2,2)} =& L^{W(2,2)}(h,h_W) \\[-1mm]
\oplus\qquad & \\[-1mm]
\langle v_p^- \rangle_{W(2,2)} \cong& L^{W(2,2)}(h+p,h_W) \\[-1mm]
\oplus\qquad & \\[-1mm]
\vdots \qquad & \\[-1mm]
\oplus\qquad & \\[-1mm]
\langle v_{ip}^- \rangle_{W(2,2)} \cong& L^{W(2,2)}(h+ip,h_W) \\[-1mm]
\oplus\qquad & \\[-1mm]
\vdots \qquad &
\end{align*}

In atypical case however, the $W(2,2)$-submodules generated by $\HVir$--singular vectors are nested as follows. Consider $V^{\HVir}(h_{p,r},h_I)$ where $(h_{p,r},h_I) \in {\mathcal {AT} }_{\HVir} ( c_L, c_{L,I} ) $. Then $\Psi^{0}$ maps a cosingular vector $u_{rp}\in V^{W(2,2)}(h_{p,r},h_W)$ to a singular vector $v_{rp}^-$. In other words we have
\begin{gather*}
\langle v_{rp}^- \rangle_{W(2,2)} \subseteq \langle v_{h_{p,r},h_I} \rangle \cong_{W(2,2)} \widetilde{L}^{W(2,2)}(h_{p,r},h_W).
\end{gather*}
Using the same argument in view of Proposition \ref{svi cosing} we see that
\begin{gather*}
\langle v_{(r-i)p}^- \rangle_{W(2,2)} \subseteq \langle v_{ip}^- \rangle_{W(2,2)} \cong \widetilde{L}^{W(2,2)}(h_{p,r}+ip,h_W), \qquad i=1,\ldots,\big\lfloor \frac{r-1}{2}\big\rfloor .
\end{gather*}
Therefore,
\begin{alignat*}{3}
&\langle v_{h_{p,r},h_I} \rangle_{W(2,2)} / \langle v_{rp}^- \rangle_{W(2,2)} \cong L^{W(2,2)}(h_{p,r},h_W), &&& \\
&\langle v_{ip}^- \rangle_{W(2,2)} / \langle v_{(r-i)p}^- \rangle_{W(2,2)} \cong L^{W(2,2)}(h_{p,r} + ip,h_W), \qquad & &i < \frac{r-1}{2},& \\
&\langle v_{ip}^- \rangle_{W(2,2)} \cong L^{W(2,2)}(h_{p,r} + ip,h_W),\qquad &&i \geq \frac{r-1}{2}.&
\end{alignat*}
In this case, $I(-p)^{r-i}v_{h_{p,r},h_I} $ are $W(2,2)$-cosingular vectors in $V^{\HVir}(h_{p,r},h_I)$.

\begin{Example} \label{verma_prim}
Consider $p=1$ case. Singular vector in $V^{\HVir}(h,0) $ is $u'_{1}=\big( L(-1) +\frac{h}{c_{L,I}}I(-1) \big) v_{h,0}$, and $u'_{1}$ generates a copy of $V^{\HVir}(h+1,0) $.

$r=1$: $\Psi \colon V^{W(2,2)}(0,0) \rightarrow V^{\HVir}(0,0) $ maps a singular vector $u'_1 = W(-1) v_{0,0}$ to $0$ and a~cosingular vector $u_1 = L(-1) v_{0,0}$ to $\HVir$-singular vector $v_1^- = L(-1) v_{0,0}$. We get the short exact sequence of $W(2,2)$-modules
\begin{gather*}
0\rightarrow L^{W(2,2)}(0,0) \rightarrow L^{\HVir}(0,0) \rightarrow L^{W(2,2)}(1,0) \rightarrow 0,
\end{gather*}
which is an expansion of (\ref{Psi}) considered as a homomorphism of $W(2,2)$-modules. The rightmost module is generated by a projective image of~$I(-1)v_{0,0}$. Therefore, $L^{\HVir}(c_L, c_{L,I})$ is generated over $W(2,2)$ by $v_{0,0}$ and $I(-1)v_{0,0}$.

$r\in \N$: In general, a cosingular vector $u_{rp}\in V^{W(2,2)}\left( \frac{1-r}{2},0\right) $ maps to a singular vector $v_r^- \in V^{
\HVir}\big( \frac{1-r}{2},0\big)$ of weight ${\frac{1+r}{2}}$.
\begin{gather*}
v_r^- = \prod\limits_{i=0}^{r-1}\left( L(-1) +\frac{1-r+2i}{2c_{L,I}}I(-1) \right) v_{\frac{1-r}{2},0}.
\end{gather*}
\end{Example}

\section[Screening operators and $W(2,2)$-algebra]{Screening operators and $\boldsymbol{W(2,2)}$-algebra}\label{screenings}

We think that the vertex algebra $L^{W(2,2)} (c_L, c_W) $ is a very interesting example of non-rational vertex algebra, which admits similar fusion ring of representations as some ${\mathcal W}$-algebras appearing in LCFT (cf.\ \cite{A1,AdM1,CM,FGST}). Since $\mathcal W$-algebras appearing in LCFT are realized as kernels of screening operators acting on certain modules for Heisenberg vertex algebras, it is natural to ask if $L^{W(2,2)} (c_L, c_W) $ admits similar realization. In~\cite{nas} we embedded the $W(2,2)$-algebra as a subalgebra of the Heisenberg--Virasoro vertex algebra. In this section we shall construct a~screening operator $S_1$ such that the kernel of this operator is exactly $L^{W(2,2)} (c_L, c_W) $.

Let us f\/irst construct a non-semisimple extension of the vertex algebra $L^{\HVir} (c_L, c_{L,I})$. Recall that the Lie algebra $\HVir$ admits the triangular decomposition~(\ref{triangular}). Let $E=\operatorname{span}_{\C} \{ v^{0}, v^{1} \}$ be $2$-dimensional $\HVir^{\ge 0} = \HVir^{0} \oplus \HVir^{+}$-module such that $\HVir ^+ $ acts trivially on $E$ and
\begin{gather*}
L(0) v^{i} = v^{i} , \qquad i=0,1, \qquad I(0) v^1 = v^0, \qquad I(0) v^0 = 0, \\
C_L v^{i} = c_L v^{i}, \qquad C_{L,I} v^{i} = c_{L,I} v^{i}, \qquad C_I v^{i} = 0, \qquad i=1,2.
\end{gather*}
Consider now induced $\HVir$-module
\begin{gather*}
\widetilde E = U(\HVir) \otimes _{U (\HVir ^{\ge 0} )} E.
\end{gather*}
By construction, $\widetilde E$ is a non-split self-extension of the Verma module $V^{\HVir} (1,0)$:
\begin{gather*}
0 \rightarrow V^{\HVir} (1,0) \rightarrow \widetilde E \rightarrow V^{\HVir} (1,0 ) \rightarrow 0.
\end{gather*}
Moreover, $\widetilde E$ is a restricted module for $\HVir$ and therefore it is a module over vertex operator algebra $L^{\HVir} (c_L, c_{L,I})$. Since
\begin{gather*}
\widetilde E \cong E \otimes U(\HVir^{-})
\end{gather*}
as a vector space, the operator $L(0)$ def\/ines a $\Z_{\ge 0}$-gradation on $\widetilde E$.

Note that $ (L(-1) + I(-1) / c_{L,I} ) v_0$ is a singular vector in $\widetilde E$ and it generates the proper submodule. Finally we def\/ine the quotient module
\begin{gather*}
\mathcal U = \frac{ \widetilde E} {U(\HVir) . (L(-1) + I(-1) / c_{L,I} ) v_0}.
\end{gather*}

\begin{Proposition}
$\mathcal U$ is a $\Zp$-graded module for the vertex operator algebra $L^{\HVir} (c_L, c_{L,I} )$:
\begin{gather*}
\mathcal U = \bigoplus _{m \in {\Zp} } \mathcal U (m), \qquad L(0) \vert \mathcal U(m) \equiv (m+1) \operatorname{Id}.
\end{gather*}
The lowest component $\mathcal U (0) \cong E$. Moreover, $\mathcal U$ is a non-split extension of the Verma module $V^{\HVir} (1,0)$ by the simple highest weight module $L^{\HVir} (1,0)$:
 \begin{gather*}
 0 \rightarrow L^{\HVir} (1,0) \rightarrow \mathcal U \rightarrow V^{\HVir} (1,0 ) \rightarrow 0.
 \end{gather*}
 \end{Proposition}
 \begin{proof}
By construction $\mathcal U $ is a graded quotient of a $\Z_{\ge 0}$-graded $L^{\HVir}(c_L, c_{L,I})$-module $\widetilde E$. The lowest component is $\mathcal U (0) \cong E$. Submodule $U(\HVir).v^0$ is isomorphic to $L^{\HVir}(1,0)$, and the projective image of $v^1$ generates the Verma module $V^{\HVir}(1,0)$ since $I(0)v^1=v^0$. For the same reason, this exact sequence does not split.
\end{proof}

Now we consider $L^{\HVir}(c_L, c_{L,I} )$-module
\begin{gather*}
\mathcal V^{\rm ext}:= L^{\HVir}(c_L, c_{L,I} ) \oplus \mathcal U.
\end{gather*}
By using \cite[Theorem 4.8.1]{LL} (see also \cite{AdM-selecta,Li1}) we have that $ \mathcal V^{\rm ext}$ has the structure of a~vertex operator algebra with vertex operator map $Y_{\rm ext}$ def\/ined as follows:{\samepage
\begin{gather*}
Y_{\rm ext} ( a_1 + w_1 , z ) (a_2 + w_2) = Y(a_1, z) (a_2+w_2) + e^{z L(-1)} Y (a_2 ,-z ) w_1,
\end{gather*}
where $a_1, a_2 \in L^{\HVir}(c_L, c_{L,I} )$, $ w_1, w_2 \in \mathcal U$.}

Take now $v^ i \in E \subset \mathcal U$, $i=0,1$ as above and def\/ine
\begin{gather*}
S_i (z) = Y_{\rm ext} \big(v^i, z\big) = \sum _{n \in \Z} S_i (n) z^{-n-1}.
\end{gather*}
By construction
\begin{gather*}
S_1 (z) \in \operatorname{End} \big( L^{\HVir} (c_L,c_{L,I}), L^{\HVir}(1,0) \big) ((z)).
\end{gather*}

\begin{Proposition}
For all $n,m\in \Z$ we have:
\begin{gather*}
[L(n), S_i (m)] = - m S_i (n+m), \qquad i=0,1,\\
[W(n), S_0 (m) ] = 0, \qquad [W(n), S_1(m) ] = 2 m c_{L,I} S_0( n+m).
\end{gather*}
In particular, $S_0 (0) $ and $S_1(0)$ are screening operators. Moreover,
\begin{gather*}
S_1= S_1(0) \colon \ L^{\HVir} (c_L, c_{L,I}) \rightarrow L^{\HVir} (1,0)
\end{gather*}
is nontrivial and $ S_1 (0) I(-1) {\bf 1} = - v_0. $
\end{Proposition}

\begin{proof} Since $L(k) v^{i} = \delta_{k,0} v^{i}$ for $k \ge 0$, commutator formula gives that
\begin{gather*}
[L(n), S_i(m)] = - m S_i (n+m).
\end{gather*}
Next we calculate $[W(n), S_1 (m)]$. We have
\begin{gather*}
W(-1) v^1 = 2 I(-1) v^0 = - 2 c_{L,I} L(-1) v^0, \\
W(0) v^1 = - 2 c_{L,I} v^ 0, \qquad W(n) v^1 = 0, \qquad n \ge 0.
\end{gather*}
This implies that
\begin{gather*}
[W(n), S_1 (m) ] = 2 c_{L,I} m S_0 (n+m).
\end{gather*}
Since $W(n) v^0 = 0$ for $n \ge -1$ we get
\begin{gather*}
[ W(n), S_0 (m) ] = 0.
\end{gather*}
Therefore we have proved that $S_i(0)$, $i= 0,1$ are screening operators. Next we have
\begin{gather*}
S_1(0) I(-1) {\bf 1} = \operatorname{Res}_z Y_{\rm ext}\big(v^1, z\big) I(-1) = \operatorname{Res}_z e ^{z L(-1)} Y( I(-1) {\bf 1}
,- z) v^1 = - v_0.
\end{gather*}
The proof follows.
\end{proof}

\begin{Theorem}
$S_1$ is a derivation of the vertex algebra $\mathcal V^{\rm ext}$ and we have
\begin{gather*}
\Ker_{ L^{\HVir} (c_L,c_{L,I}) } S_1 \cong L ^{W(2,2)} (c_L, c_W ).
\end{gather*}
\end{Theorem}
\begin{proof}
By construction $S_1 = \operatorname{Res}_z Y_{\rm ext} (v^1, z)$, so $S_1$ is a derivation so $\overline W = \Ker_{ L^{\HVir} (c_L,c_{L,I}) } S_1$ is a vertex subalgebra of $L^{\HVir} (c_L,c_{L,I})$. Since
\begin{gather*}
S_1 L(-2) {\bf 1} = S_1 W(-2) {\bf 1} = 0
\end{gather*}
we have that $L ^{W(2,2)} (c_L, c_W ) \subset \overline W$. Since $S_1 I (-1) {\bf 1} \ne 0$, we have that $I(-1) {\bf 1}$ does not belong to $\overline W$. By using the fact that $L^{\HVir} (c_L,c_{L,I})$ is as $W(2,2)$-module generated by singular vector $\bf{1}$ and cosingular vector $I(-1) {\bf 1}$ (see Example~\ref{verma_prim}) we get that $\overline W = L ^{W(2,2)} (c_L, c_W )$. The proof follows.
\end{proof}

\begin{Remark}
Of course, every $\mathcal V_{\rm ext}$-module becomes a $W(2,2)$-module with screening opera\-tor~$S_1$. Similar statement holds for intertwining operators. Constructions of such modules and intertwining operators require dif\/ferent techniques which we will present in our forthcoming paper \cite{AR2}.
\end{Remark}

\subsection*{Acknowledgements}
The authors are partially supported by the Croatian Science Foundation under the project 2634 and by the Croatian Scientif\/ic Centre of Excellence QuantixLie.

\pdfbookmark[1]{References}{ref}
\LastPageEnding

\end{document}